\documentclass[11pt,reqno,a4paper]{amsart}

\usepackage{amsmath,amsfonts,amssymb,amsthm,mathrsfs}
\usepackage{rotating,pgf,tikz,float}
\usepackage{hyperref,graphicx,color}
\usetikzlibrary{automata,positioning,arrows.meta,shapes.misc}

\usepackage[normalem]{ulem}
\usepackage{indentfirst}
\usepackage[hmargin=3.0cm,bmargin=3.0cm,tmargin=3.5cm]{geometry}

\newtheorem{theorem}{Theorem}[section]
\newtheorem{lemma}[theorem]{Lemma}
\newtheorem{proposition}[theorem]{Proposition}
\newtheorem{corollary}[theorem]{Corollary}

\newtheorem*{conjecture*}{Conjecture}

\theoremstyle{definition}
\newtheorem{problem}[theorem]{Open Problem}

\theoremstyle{remark}
\newtheorem{remark}[theorem]{Remark}
\newtheorem{definition}[theorem]{Definition}
\newtheorem{example}[theorem]{Example}

\numberwithin{equation}{section}
\numberwithin{figure}{section}

\newcommand{\R}{\mathbb{R}}

\newcommand{\dx}{{\mathrm d}x}
\newcommand{\dy}{{\mathrm d}y}
\renewcommand{\d}{{\mathrm d}}

\DeclareMathOperator{\dist}{dist}

\DeclareMathOperator{\diameter}{diam}

\newcommand\Graph{\mathcal{G}}
\newcommand\HGraph{\mathcal{H}}

 \newcommand\mV{\mathsf{V}}
 
 \newcommand\mE{\mathsf{E}}
 \newcommand\mv{\mathsf{v}}

\title{Mean distance on metric graphs}



\author{Lu\'is N.\ Baptista}
\author{James B.\ Kennedy}
\author{Delio Mugnolo}

\address{Departamento de Matem\'atica, Faculdade de Ci\^encias da Universidade de Lisboa,
Campo Grande, Edif\'icio C6, 1749-016 Lisboa, Portugal {\rm and} Grupo de F\'isica Matem\'atica, Instituto Superior T\'ecnico, Av.\ Rovisco Pais, 1049-001 Lisboa, Portugal}
\email{lcbaptista@ciencias.ulisboa.pt}
\address{Departamento de Matem\'atica, Faculdade de Ci\^encias da Universidade de Lisboa,
Campo Grande, Edif\'icio C6, 1749-016 Lisboa, Portugal {\rm and} Grupo de F\'isica Matem\'atica, Instituto Superior T\'ecnico, Av.\ Rovisco Pais, 1049-001 Lisboa, Portugal}
\email{jbkennedy@ciencias.ulisboa.pt}
\address{Lehrgebiet Analysis, Fakult{\"a}t Mathematik und Informatik, FernUniversit{\"a}t in Hagen, D-58084 Hagen, Germany}
\email{delio.mugnolo@fernuni-hagen.de}

\date{\today}

\thanks{L.N.B. and J.B.K. were supported by the Funda\c{c}\~ao para a Ci\^encia e a Tecnologia, Portugal, reference  PTDC/MAT-PUR/1788/2020 (J.B.K.) and UIDB/00208/2020 (both authors). D.M.\ was partially supported by the Deutsche Forschungsgemeinschaft (Grant 397230547).  This article is based upon work from COST Action 18232 MAT-DYN-NET, supported by COST (European Cooperation in Science and Technology), \url{www.cost.eu}.\\
The authors are very grateful to Pavel Kurasov (Stockholm) for participating in the early discussions from which this article developed, { and to Noah Kravitz (Princeton) for a number of very helpful observations and suggestions, including the idea behind Proposition~\ref{prop:pendant-derivative}.}}

\begin{document}

\begin{abstract}
We introduce a natural notion of mean (or average) distance in the context of compact metric graphs, and study its relation to geometric properties of the graph. We show that it exhibits a striking number of parallels to the reciprocal of the spectral gap of the graph Laplacian with standard vertex conditions: it is maximised among all graphs of fixed length by the path graph (interval), or by the loop in the restricted class of doubly connected graphs, and it is minimised among all graphs of fixed length and number of edges by the equilateral flower graph. We also establish bounds for the correctly scaled product of the spectral gap and the square of the mean distance which depend only on combinatorial, and not metric, features of the graph. This raises the open question whether this product admits absolute upper and lower bounds valid on all compact metric graphs.
\end{abstract}

\maketitle

\section{Introduction}
\label{sec:introduction}

Our aim is to study the notion of mean, or average, distance on metric graphs. Metric graphs have been studied in mathematics for over 40 years, at least since~\cite{Lum80}. They consist of collections of intervals whose endpoints are suitably identified to create a connected network structure, and display many features which mirror those of combinatorial, i.e.\ discrete, graphs; but at times they exhibit more complex behaviour reminiscent of (higher dimensional) Euclidean domains or manifolds, thanks to their continuous metric structure. In particular, it is possible to define differential operators such as Laplacians on metric graphs, see~\cite{BerKuc13,Kur23,Mug14}; in the last decade or so, the interplay between the geometry of the graph and the Laplacian spectrum has received considerable attention, see, e.g., \cite{BerKenKur19,Kur23} and the references therein.

At the same time, mean distance does not previously seem to have been considered on metric graphs. On \emph{combinatorial} graphs a corresponding notion was already actively investigated in the late 20th Century; in fact, it seems difficult to determine a precise birthdate of the theory, as there the mean distance is, up to a normalising factor, simply the $\ell^1$-norm of the distance matrix. According to \cite{EntJacSny76}, this and related quantities were first studied by Harary~\cite{Har59} and Ore~\cite{Ore62} as graph-based sociometric quantities; while in \cite{GodOel11} the authors trace the origin of this notion back to an even earlier study in quantum chemistry~\cite{Wie47}. Subsequently, mean distance has grown to be an important quantity in combinatorial geometry ever since \cite{DoyGra77,EntJacSny76,Ple84}, and it has also played an important role in inverse combinatorial problems~\cite{Chu88}. {The interplay of the $\ell^1$-norm of the distance matrix with the geometry and potential theory of graphs was investigated more recently, in \cite{Ste23}; in particular, a new notion of curvature is suggested that, in particular, agrees with the mean distance function whenever the latter is vertex-wise constant. As a further development that is closer to our main interest in this article, let us mention that o}ne topic of interest since at least the 1970s has been to study the relationship between the combinatorial mean distance and the eigenvalues of the discrete graph Laplacian, see, e.g., \cite{GodOel11,Moh91b}, and see below.

Our goals here are essentially twofold. First, we will develop fundamental geometric bounds for the mean distance of a (compact) metric graph: on the one hand, it seems that the mean distance is another, and arguably quite fine, quantity that measures how well connected the graph is, just like the Cheeger constant and the vertex connectivity do (as our estimates will show). On the other, in practice it is somewhat difficult to compute, thus providing a need for such estimates. What we will show (see Theorem~\ref{thm:main}) is that the mean distance is maximised among all graphs of given total length when the graph is a path, i.e.\ interval (or a loop among doubly connected graphs); and maximised among all graphs of given total length and number of edges when the graph is a so-called equilateral flower graph.

These results have strong parallels with both the behaviour of mean distance on combinatorial graphs, and the spectral gap $\mu_2$ of the standard Laplacian on metric graphs. On combinatorial graphs it has been known for decades that mean distance is minimised on complete graphs \cite[Theorem~2.3]{EntJacSny76} and maximised on path graphs \cite[Corollary~1.3]{DoyGra77}, or on cycles among doubly-connected graphs \cite[Theorem~6]{Ple84}. This in turn mirrors inversely what is known regarding the spectral gap of the combinatorial Laplacian, that is, the algebraic connectivity: it is famously minimised by path graphs and maximised by complete graphs, see \cite{Fiedler}. This strong inverse relationship in the discrete world has been studied: for example, in~\cite{Moh91b}, quite sophisticated estimates on the lowest positive eigenvalue of the discrete Laplacian are derived in terms of the mean distance in combination with the  size of the graph 
% (i.e., the number of its vertices)
and its maximal degree (see also \cite{Moh91c}); as a more recent contribution to this field we mention \cite{Siv09}.

Regarding the spectral gap $\mu_2$ of the standard Laplacian on metric graphs, it is known that $\mu_2$ is minimised on the path graph \cite[Th\'eor\`eme 3.1]{Nic87}, or the cycle among doubly-connected graphs \cite[Theorem~2.1]{BanLev17}, and maximised, under a constraint on the number of edges, among others by the equilateral flower graph \cite[Theorem~4.2]{KenKurMal16}. Thus, in particular, the strong correspondence in the discrete setting also seems to hold in the metric one.

Our second goal is thus to investigate precisely the relationship between the mean distance $\rho$ and the spectral gap $\mu_2$. We will provide several upper and lower bounds on the product $\mu_2 \rho^2$ { (scaled in such a way as to be independent of the total length of the graph)}, see Theorem~\ref{thm:direct-lower-bound} and Corollaries~\ref{cor:comparison} and~\ref{cor:mu-rho-lower-bound}, which depend only on the number of edges of the graph and/or its (first) Betti number (number of independent cycles); in the case of trees one can find an absolute constant as an upper bound, $\pi^2$. These bounds in general do not appear to be optimal, and the natural question arises as to whether the product $\mu_2 \rho^2$ admits absolute upper and lower bounds valid for \emph{any} compact graph, which we leave as an open problem (Open Problem~\ref{problem:absolute}).

This paper is structured as follows. In Section~\ref{sec:notation} we will introduce our general framework and provide a definition of mean distance $\rho$ on a metric graph. All our main results -- geometric estimates for $\rho$ as well as bounds linking $\rho$ and the spectral gap of the standard Laplacian -- are collected in Section~\ref{sec:main}. In addition to Theorem~\ref{thm:main}, we also include a statement on the relationship between $\rho$ and the diameter of the graph, Theorem~\ref{thm:diameter}.

The proofs, which rely on various techniques, are spread throughout the remaining sections: first, we give several explicit examples in Section~\ref{sec:examples}, namely star graphs (including intervals, see Example~\ref{ex:star}), a special type of ``firework graph'' which shows that mean distance can be arbitrarily close to diameter (Example~\ref{ex:firework}), and flower graphs (including loops, Example~\ref{ex:flower}), whose analysis is essential for one of the main geometric bounds in Theorem~\ref{thm:main}.

Then, in Section~\ref{sec:surgery}, we study \emph{surgery principles} for mean distance, which will be used to prove Theorem~\ref{thm:main}. These principles are inspired by analogous ones developed in recent years for spectral and torsional quantities on metric graphs (see in particular \cite{BerKenKur19}), but, curiously, such surgical methods also play a role in some of the early combinatorial works; we refer to~\cite[Sections~5.3 and 5.4]{GodOel11} for an overview of results in this direction. However, while there are parallels to surgery for spectral quantities, in concrete terms and in the details our results and proofs are quite different as there is currently no known variational characterisation of mean distance (cf.\ Remark~\ref{rem:no-variation}); instead, one needs to analyse more directly the effect of altering a graph on its distance function. { There are also notable differences in some results; attaching a so-called pendant edge to a graph has an indeterminate effect on $\rho$, despite always lowering $\mu_2$ (see Proposition~\ref{prop:pendant-derivative}).}

In Section~\ref{sec:symmetrisation} we will provide the proof that among doubly connected graphs of given length the mean distance is maximised on the loop, the last remaining part of Theorem~\ref{thm:main}. Here, somewhat surprisingly, it is natural to use a symmetrisation argument based on the coarea formula and inspired by the approach of Friedlander for minimising standard Laplacian eigenvalues \cite{Fri05}; however, here again, the actual technical nature of the proof is very different as mean distance is not characterised as a minimising quantity.

Finally, in Section~\ref{sec:test}, we give a proof of the lower bound in Theorem~\ref{thm:direct-lower-bound} on the product of $\rho$ and the spectral gap using variational methods for the latter.

Before continuing, we mention the very recent paper~\cite{GarMarSil23}, which we became aware of shortly before submitting the present paper: seemingly unaware of the theory of metric graphs, the authors of~\cite{GarMarSil23} introduce a notion of mean distance that agrees with ours and discuss the complexity of its computing.

We note that the notion of mean distance is in fact a natural one in any \textit{metric measure space} of finite diameter and volume. In the spirit of~\cite{Dan12}, it would be natural to introduce the mean distance of further, possibly more general metric measure spaces: an interesting estimate on the lowest positive eigenvalue of the Laplace--Beltrami operator of a compact Riemannian manifold was already announced, but not proved, in~\cite{Moh89}. More modestly, one could try and extend the scope of our results to metric graphs of finite total length but an infinite number of edges. Since many of our methods are most natural in the compact case of finitely many edges, and tailored to graphs, to reduce technical complications and keep the work self-contained we will restrict to this case and not consider such possible generalisations here.

\section{Notation and assumptions}
\label{sec:notation}

Throughout, we consider metric graphs $\Graph:=(\mV,\mE,\ell)$ consisting of a finite number $V:=\#\mV$ of vertices and $E:=\#\mE$ of edges; we denote by $\beta = E - V + 1$ the (first) Betti number of $\Graph$, the number of independent cycles it contains. 
%: we denote the vertex and edge set by $\mathcal V$ and $\mathcal E$, respectively.

Unlike in the case of combinatorial graphs, we create a metric structure by associating each edge $e$ with a real interval $(0,\ell_e)$ of length $\ell_e>0$, which we may regard as a parametrisation of the edge. While this parametrisation implies an orientation of each edge, all quantities considered will be independent of these orientations. 
%; it is sometimes convenient to consider the set of such metric edges,
%\[
%\mathcal E:=\bigsqcup_{e\in\mE} (0,\ell_e)
%\]
In particular, the \textit{total length} $|\Graph|$ of $\Graph$ is, by definition
\[
|\Graph|:=\sum_{e\in\mE} \ell_e.
\]
A distance on $\Graph$ is introduced by first considering the Euclidean distance on each metric edge and then extending it to a (pseudo-)metric $\dist$ on $\Graph$ by the usual shortest-path construction (we will also sometimes write $\dist_\Graph$ in those cases where it becomes necessary to specify the graph).

To avoid trivialities, \textit{we will always assume without further comment that the metric graph is connected}. A metric graph then canonically becomes a metric measure space $(\Graph,\dist,\dx)$ upon endowing each metric edge with the 1-dimensional Lebesgue measure $\dx$. This metric measure structure immediately induces the spaces $C(\Graph)$ and $L^2(\Graph)$ of continuous and square integrable functions over $\Graph$, respectively. We will also need the Sobolev space
\[
H^1(\Graph):=\{f\in C(\Graph)\cap L^2(\Graph):f'\in L^2(\Graph)\}.
\]
As is well known, all these spaces are, up to a canonical identification, independent of the choice of parametrisation of the edges, and in particular of the corresponding orientation.

Our assumptions, namely that there are only finitely many edges each of finite length, imply that as a metric space $\Graph$ is compact and in particular has finite diameter. In this case, as is standard, we will refer to $\Graph$ as a \emph{compact metric graph}. (We refer to~\cite{BerKuc13,Mug19} for a more detailed introduction to the notion of metric graphs.)

Our compactness assumption on $\Graph$ implies that the following natural notion of mean distance is well defined.

\begin{definition}\label{defi:meandist}
The \textit{mean distance function} on $\Graph$ is the function $\rho_\Graph : \Graph \to \R$ defined by
\[
%\rho_\Graph(x):=\frac{1}{\mu(\Graph\setminus\{x\})}\int_{\Graph\setminus\{x\}}\dist(x,y) \dy.
\rho_\Graph(x):=\frac{1}{|\Graph|}\int_{\Graph}\dist(x,y)\, \dy;
\]
we will call $\rho_\Graph(x)$ the \emph{mean distance from $x$} on $\Graph$. The \textit{mean distance} on $\Graph$ is then defined as the mean value of $\rho_\Graph$,
\[
\rho(\Graph):=\frac{1}{|\Graph|}\int_{\Graph}\rho_\Graph(x)\, \dx =\frac{1}{|\Graph|^2}\int_{\Graph}\int_{\Graph}\dist(x,y)\,\dy\,\dx .
\]
\end{definition}

(As mentioned in the introduction, this concept has been studied for decades in graph theory, where it is commonly referred to as either ``mean distance'' or ``average distance'' in the literature; we will always use the former.)

It would be possible to extend this definition to graphs of finite total length but an infinite number of edges; however, as many of our techniques are naturally adapted to the compact case, and to keep the exposition more simple, we will not do so.

Beyond estimates on the mean distance of a graph $\Graph$ in terms of quantities such as the total length, the number of edges and the Betti number of the graph, we will be particularly interested in the interplay between the mean distance and the spectrum of the corresponding Laplacian $-\Delta_\Graph$ with standard vertex conditions, in particular as regards its \textit{spectral gap} (see, e.g., \cite{BerKenKur19,BerKuc13}). For our purposes the following characterisation will suffice.

\begin{definition}
The \textit{spectral gap} of $\Graph$ is
\begin{equation}\label{eq:spectral-rayleigh}
\mu_2(\Graph):=\inf_{\substack{{u\in H^1(\Graph)}\\{\int_\Graph u(x)\dx =0}}}\frac{\|u'\|^2_{L^2(\Graph)}}{\|u\|^2_{L^2(\Graph)}}.
\end{equation}
\end{definition}
This quantity is known to be strictly positive since $\Graph$ is compact and connected.

\section{Main results}
\label{sec:main}

We start with a simple observation on the relationship between $\rho(\Graph)$ and the diameter $\diameter(\Graph)$; completely analogous results have been known to hold for discrete graphs since the 1970s (see \cite{DoyGra77}). While the (strict) inequality is rather obvious, the key point is that the inequality is sharp nonetheless.

\begin{theorem}
\label{thm:diameter}
Let $\Graph$ be a compact metric graph of diameter $\diameter (\Graph) > 0$. Then
\begin{equation}
\label{eq:rho-diameter}
    \rho (\Graph) < \diameter (\Graph).
\end{equation}
Moreover, there exists a sequence of graphs $\Graph_n$ for which $\frac{\rho (\Graph_n)}{\diameter (\Graph_n)} \to 1$ as $n \to \infty$.
\end{theorem}

%``Firework'' example: $\rho (\Graph_n) \to$ diameter. (Note there is a corresponding known example for discrete graphs in the old paper of Doyle--Graver \cite{DoyGra77}.)

\begin{proof}
For \eqref{eq:rho-diameter}, by definition of $\rho$ and Cauchy--Schwarz,
%\[
%\|\dist\|_1 ^2\le \|1\|^2_2 \|\dist \|^2_2=|\Graph|^2 \|\dist \|^2_2
%\]
%hence
\begin{equation}
\label{eq:holdbasic}
\rho(\Graph)=\frac{1}{|\Graph|^2}\|\dist\|_1< \frac{1}{|\Graph|}\|\dist\|_2\le\|\dist\|_\infty= \diameter(\Graph),
\end{equation}
the \textit{strict} inequality being due to the linear independence of $\dist$ and $\mathbf 1$. 
%The inequality \eqref{eq:rho-diameter} is a trivial consequence of the fact that $\dist (x,y) \leq \diameter (\Graph)$ for all $x,y \in \Graph$ and the definition of $\rho$ (see \eqref{eq:holdbasic}). 
For a sequence of graphs $\Graph_n$ for which $\frac{\rho (\Graph_n)}{\diameter (\Graph_n)} \to 1$ we will consider ``firework'' graphs similar to the (discrete) ones used in \cite{DoyGra77} to prove a corresponding assertion for the discrete mean distance; these will be given in Example~\ref{ex:firework} below.
\end{proof}

We now turn to our first main result, which consists of fundamental upper and lower bounds on $\rho$ in terms only of the length, and the length and number of edges, of the graph $\Graph$, respectively. As discussed in the introduction, these closely mirror those for the (inverse of the) first nontrivial standard Laplacian eigenvalue $\mu_2(\Graph)$: this eigenvalue admits a lower bound depending only on length \cite[Th\'eor\`eme 3.1]{Nic87}, with equality for path graphs (i.e. intervals); there is an improved lower bound for \emph{doubly connected graphs} depending only on length and with equality for loops \cite[Theorem~2.1]{BanLev17}. There is a corresponding upper bound based only on length and number of edges, with equality for equilateral flower graphs, among others \cite[Theorem 4.2]{KenKurMal16}. Our result is:

\begin{theorem}
\label{thm:main}
Let $\Graph$ be a compact metric graph with total length $|\Graph|=L>0$ and $E$ edges. Then
\begin{equation}
\label{eq:bounds-1}
	\left(\frac{2E-1}{4E^2}\right)L \leq \rho (\Graph) \leq \frac{L}{3},
\end{equation}
with equality in the lower bound if and only if $\Graph$ is an equilateral flower graph, and equality in the upper bound if and only if $\Graph$ is a path graph (i.e. interval). If $\Graph$ is doubly connected, then the upper bound may be improved to
\begin{equation}
\label{eq:bounds-2}
	\rho (\Graph) \leq \frac{L}{4},
\end{equation}
with equality if and only if $\Graph$ is a loop.
\end{theorem}

For \eqref{eq:bounds-1} we will use \emph{surgery} methods, which are the subject of Section~\ref{sec:surgery}; the proof of \eqref{eq:bounds-1} is given at the end of that section. To prove \eqref{eq:bounds-2} we will use a symmetrisation-type argument based on the coarea formula (which, curiously, does not seem to work in the base case), see Section~\ref{sec:symmetrisation}.

Since $\rho$ and $\mu_2$ satisfy compatible bounds, we can also find two-sided bounds for the scale invariant quantity $\mu_2(\Graph) \rho(\Graph)^2$:

\begin{corollary}
\label{cor:comparison}
Let $\Graph$ be a compact metric graph with $E$
edges and first Betti number $\beta \geq 0$. Then
\begin{equation}
\label{eq:comparison}
		\frac{\pi^2(2E-1)^2}{16E^4}\leq \mu_2(\Graph) \rho(\Graph)^2 \leq \min \left\{ \frac{\pi^2 E^2}{9}, \pi^2 (1+\beta)^2\right\}
\end{equation}
\end{corollary}

The fact that the respective minimisers and maximisers are in perfect correspondence, and more generally that the two quantities behave so analogously, suggests that there may actually be absolute upper and lower bounds, i.e., absolute constants independent of every graph which control $\mu_2(\Graph) \rho(\Graph)^2$ from above and below. At least for \emph{trees}, \eqref{eq:comparison} gives an absolute upper bound on this quantity (namely $\min \{\frac{\pi^2 E^2}{9}, \pi^2 \}$), which is sharp for path graphs $I$, where $\mu_2(I)\rho(I)^2 = \frac{\pi^2}{9}$. However, a direct computation using the values obtained in Section~\ref{sec:examples} shows that both for equilateral flowers and stars the product $\mu_2\rho^2$ is monotonically increasing and approaches $\frac{\pi^2}{4}$ as $E\to \infty$.

We note that there is a certain parallel between the mean distance of a metric graph and (upon imposing a Dirichlet condition on at least one vertex) its torsional rigidity, as the latter is known, too, to be maximised precisely by path graphs and minimised precisely by equilateral flowers, see~\cite[Section~4]{MugPlu23}: this behaviour is precisely opposite to that of the lowest positive Laplacian eigenvalue, so it seems natural to pose the following, once appropriate scaling is accounted for.

\begin{problem}
\label{problem:absolute}
Do there exist absolute constants $C,c>0$ such that
\begin{equation}\label{eq:polya?}
    c \leq \mu_2(\Graph) \rho(\Graph)^2 \leq C
\end{equation}
for \emph{all} compact graphs $\Graph$?
\end{problem}

Indeed, the conjectured lower bound recalls the celebrated Kohler-Jobin inequality, which asserts that, in $\R^d$, the normalised quantity $\lambda_1(\Omega) \tau(\Omega)^\frac{2}{d+2}$ (with $\lambda_1$ the first Dirichlet Laplacian eigenvalue and $\tau$ the torsional rigidity of $\Omega$) is minimised when $\Omega$ is a ball (see \cite{Koh78} or, e.g., \cite{Bra14,BriButPri22} for more recent results in the area, as well as \cite[Theorem~5.8]{MugPlu23} for the metric graph case). This, as well as our examples in Section~\ref{sec:examples}, provide some weak evidence suggesting that the path graph might actually minimise $\mu_2\rho^2$, which, if true, would result in an optimal value of $c=\frac{\pi^2}{9}$.

\begin{proof}[Proof of Corollary~\ref{cor:comparison}]
For the lower bound combine the lower bound in \eqref{eq:bounds-1} with Nicaise' inequality \cite[Th\'eor\`eme~3.1]{Nic87}. For the first upper bound use the upper bound in \eqref{eq:bounds-1} and \cite[Theorem~4.2]{KenKurMal16}, for the second combine $\rho(\Graph) \leq \diameter (\Graph)$ (cf.\ Theorem~\ref{thm:diameter}) with \cite[Theorem~5.2]{DufKenMug22}.
\end{proof}

We can give an alternative lower bound on $\rho$ involving the total length and $\mu_2$, which in combination with the lower bound in Theorem~\ref{thm:main} leads to an alternative to the lower bound in Corollary~\ref{cor:comparison}:

\begin{theorem}
\label{thm:direct-lower-bound}
Let $\Graph$ be a compact metric graph with total length $|\Graph| = L>0$. Then
\begin{equation}
\label{eq:direct-comparison}
    \frac{1}{L} \leq \mu_2(\Graph) \rho(\Graph).
\end{equation}
\end{theorem}

\begin{corollary}
\label{cor:mu-rho-lower-bound}
Let $\Graph$ be a compact metric graph. Then
\begin{displaymath}
    \frac{2E-1}{4E^2} \leq \mu_2 (\Graph) \rho (\Graph)^2.
\end{displaymath}
\end{corollary}

While still unlikely to be optimal, this is generally better than the lower bound in Corollary~\ref{cor:comparison}; as $E \to \infty$ it behaves like $\frac{1}{2E}$ rather than $\frac{\pi^2}{4E^2}$. In fact, the lower bound in Corollary~\ref{cor:mu-rho-lower-bound} is seen to be larger whenever $\frac{\pi^2}{4E^2}(2E-1) < 1$, that is, $4E^2 - 2\pi^2 E + \pi^2 > 0$, which holds for all $E \geq 5$.

The proof of Theorem~\ref{thm:direct-lower-bound} is based on a suitable eigenvalue comparison argument, which involves linking the first eigenvalue of an auxiliary problem with a Dirichlet condition at a well-chosen vertex on the one hand, and a related distance function on the other, and is given in Section~\ref{sec:test}.

\section{Examples}
\label{sec:examples}

Here we will collect some examples where one can determine $\rho$ explicitly: equilateral star graphs (including path graphs, i.e. intervals, as a special case), the ``firework graphs'' that will be used to complete the proof of Theorem~\ref{thm:diameter}, and equilateral flowers (including on one edge, i.e. loops). For purposes of legibility, we will generally use $m$ and $n$ to denote the number of edges and vertices, respectively, of our example graphs.

\begin{example}[Star graph]
\label{ex:star}
For a given $m \geq 1$ and $L > 0$, consider $S_m$ an equilateral star graph, that is, a graph with $m$ edges of equal length $\ell:= \frac{L}{m}$ (and thus total length $L$) and $m+1$ vertices, one of degree $m$ (the ``central vertex'') and all others of degree $1$. To calculate $\rho$, we identify each of the $m$ identical edges with an interval of the form $[0,\frac{L}{m}]$, where $0$ corresponds to the central vertex. Then, for any $x \in S_m$,
\begin{displaymath}
\begin{aligned}
    \rho_{S_m}(x)&=\frac{1}{L} \int_{S_m} \dist(x,y) \dy \\
    &=\frac{1}{L}\left[\int_0^{\frac{L}{m}} |y-x| \dy + (m-1)\int_0^{\frac{L}{m}} \big(y+x\big) \dy\right]\\
    &=\frac{x^2}{L}+\frac{m-2}{m}x + \frac{L}{2m};
\end{aligned}
\end{displaymath}
note that this includes the special case of an interval $I$ ($m=1$, with the interval parametrised as $[0,L]$), where
\begin{equation}
\label{eq:rho-interval}
    \rho_I (x) = \frac{x^2}{L} - x + \frac{L}{2}.
\end{equation}
Integrating $\rho_{S_m}$ over all $x \in [0,\frac{L}{m}]$ (and using the symmetry of the graph) yields
\begin{displaymath}
    \rho(S_m)=\frac{L}{m^2}\left(m-\frac{2}{3}\right).
\end{displaymath}
Let us consider the behaviour of the formula as a function of $m$. If $m=1$ or $m=2$, then $\rho(S_m)=\frac{L}{3}$, which is exactly the value of $\rho$ on an interval (note that stars on one and two edges are both path graphs).
   
Now fix $L$, then since the function $x \mapsto \frac{L}{x^2}(x-\frac{2}{3})$, $x > 0$, attains its maximum at $x=\frac{4}{3}$ (and the same value at $x=1$ and $x=2$), and is monotonically decreasing for $x > \frac{4}{3}$, the maximum in $m$ is attained at $m=1,2$, and $\rho(S_m)$ is a decreasing function of $m \geq 1$. We observe that $\rho(S_m) \to 0$ as $m \to \infty$.

If instead we fix the length of each edge $\ell=\frac{L}{m}$ (and vary $m$), then
\begin{displaymath}
    \rho(S_m)=\ell-\frac{2}{3m}\ell \longrightarrow \ell
\end{displaymath}    
as $m \to \infty$.
\end{example}

We next consider a class of graphs built around stars, which include the sequence $\Graph_n$ referred to in Theorem~\ref{thm:diameter}, for which $\frac{\rho(\Graph_n)}{\diameter(\Graph_n)} \to 1$.

\begin{example}[Firework graph]
\label{ex:firework}
Here we will construct a four-parameter family of graphs, which for simplicity we will simply call $\Graph$, as follows. We start with an equilateral star on $m \geq 1$ edges, each of length $J>0$, and then attach, at each of the $m$ degree one vertices, a copy of another star $*$  with $n$ edges of equal length $j$; see Figure~\ref{fig:firework}. In particular, $|\Graph|=mJ+mnj$, and $\diameter(\Graph) = 2J + 2j$.
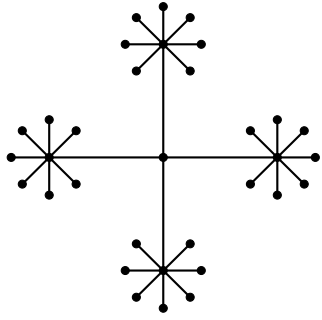
\begin{figure}[ht]
\centering
\begin{tikzpicture}
\draw[thick] (0,0) -- (0,1.5);
%\draw[thick] (0,0) -- (0.866,-0.5);
%\draw[thick] (0,0) -- (-0.866,-0.5);
\draw[thick] (0,0) -- (1.5,0);
\draw[thick] (0,0) -- (0,-1.5);
\draw[thick] (0,0) -- (-1.5,0);
\draw[fill] (0,0) circle (1.5pt);
\draw[fill] (0,1.5) circle (1.5pt);
\draw[fill] (1.5,0) circle (1.5pt);
\draw[fill] (0,-1.5) circle (1.5pt);
\draw[fill] (-1.5,0) circle (1.5pt);
\draw[thick] (0,1.5) -- (0.5,1.5);
\draw[fill] (0.5,1.5) circle (1.5pt);
\draw[thick] (0,1.5) -- (-0.5,1.5);
\draw[fill] (-0.5,1.5) circle (1.5pt);
\draw[thick] (0,1.5) -- (0,2);
\draw[fill] (0,2) circle (1.5pt);
\draw[thick] (0,1.5) -- (0.354,1.854);
\draw[fill] (0.354,1.854) circle (1.5pt);
\draw[thick] (0,1.5) -- (-0.354,1.854);
\draw[fill] (-0.354,1.854) circle (1.5pt);
\draw[thick] (0,1.5) -- (0.354,1.5-0.354);
\draw[fill] (0.354,1.5-0.354) circle (1.5pt);
\draw[thick] (0,1.5) -- (-0.354,1.5-0.354);
\draw[fill] (-0.354,1.5-0.354) circle (1.5pt);
\draw[thick] (1.5,0) -- (1.5,0.5);
\draw[fill] (1.5,0.5) circle (1.5pt);
\draw[thick] (1.5,0) -- (1.5,-0.5);
\draw[fill] (1.5,-0.5) circle (1.5pt);
\draw[thick] (1.5,0) -- (2,0);
\draw[fill] (2,0) circle (1.5pt);
\draw[thick] (1.5,0) -- (1.854,0.354);
\draw[fill] (1.854,0.354) circle (1.5pt);
\draw[thick] (1.5,0) -- (1.854,-0.354);
\draw[fill] (1.854,-0.354) circle (1.5pt);
\draw[thick] (1.5,0) -- (1.5-0.354,0.354);
\draw[fill] (1.5-0.354,0.354) circle (1.5pt);
\draw[thick] (1.5,0) -- (1.5-0.354,-0.354);
\draw[fill] (1.5-0.354,-0.354) circle (1.5pt);
\draw[thick] (-1.5,0) -- (-1.5,0.5);
\draw[fill] (-1.5,0.5) circle (1.5pt);
\draw[thick] (-1.5,0) -- (-1.5,-0.5);
\draw[fill] (-1.5,-0.5) circle (1.5pt);
\draw[thick] (-1.5,0) -- (-2,0);
\draw[fill] (-2,0) circle (1.5pt);
\draw[thick] (-1.5,0) -- (-1.854,0.354);
\draw[fill] (-1.854,0.354) circle (1.5pt);
\draw[thick] (-1.5,0) -- (-1.854,-0.354);
\draw[fill] (-1.854,-0.354) circle (1.5pt);
\draw[thick] (-1.5,0) -- (-1.5+0.354,0.354);
\draw[fill] (-1.5+0.354,0.354) circle (1.5pt);
\draw[thick] (-1.5,0) -- (-1.5+0.354,-0.354);
\draw[fill] (-1.5+0.354,-0.354) circle (1.5pt);
\draw[thick] (0,-1.5) -- (0.5,-1.5);
\draw[fill] (0.5,-1.5) circle (1.5pt);
\draw[thick] (0,-1.5) -- (-0.5,-1.5);
\draw[fill] (-0.5,-1.5) circle (1.5pt);
\draw[thick] (0,-1.5) -- (0,-2);
\draw[fill] (0,-2) circle (1.5pt);
\draw[thick] (0,-1.5) -- (0.354,-1.854);
\draw[fill] (0.354,-1.854) circle (1.5pt);
\draw[thick] (0,-1.5) -- (-0.354,-1.854);
\draw[fill] (-0.354,-1.854) circle (1.5pt);
\draw[thick] (0,-1.5) -- (0.354,-1.5+0.354);
\draw[fill] (0.354,-1.5+0.354) circle (1.5pt);
\draw[thick] (0,-1.5) -- (-0.354,-1.5+0.354);
\draw[fill] (-0.354,-1.5+0.354) circle (1.5pt);
\end{tikzpicture}
\caption{A firework graph with $m=4$, $n=7$, and $J=5j$.}
\label{fig:firework}
\end{figure}
We will show that, for the correct choice of $m, n, J, j$, $\rho(\Graph)$ can be made arbitrarily close to $\diameter(\Graph)$ in the sense of Theorem~\ref{thm:diameter}; for this we merely need to estimate $\rho (\Graph)$ from below.

First, it is obvious that, for any $x \in \Graph$ in any copy of $*$,
\begin{displaymath}
    \rho_{\Graph}(x) \geq \frac{1}{|\Graph|}2J(m-1)nj.
\end{displaymath}
Indeed, we may estimate $\dist(x,y)$ from below by $2J$ for each $y$ in each of the other $m-1$ copies of $*$; this can then be integrated over all such copies, each of which has total length $nj$. This then leads to

\begin{displaymath}
\begin{aligned}
    \frac{1}{|\Graph|}\int_{\Graph}\rho_{\Graph}(x) \dx
    & \geq  \frac{1}{|\Graph|}\int_{\cup *}\rho_{\Graph}(x)\dx \\
    & \geq \frac{1}{|\Graph|} nmj \frac{1}{|\Graph|} 2J(m-1)nj \\
    & = \frac{n^2m(m-1)j^2 2J}{(mJ + mnj)^2} %\\
    %& 
    = 2J \cdot \frac{m-1}{m} \cdot \frac{n^2j^2}{(J+nj)^2}
\end{aligned}
\end{displaymath}
and thus (cf.\ \eqref{eq:rho-diameter})
\begin{displaymath}
    \diameter(\Graph) > \rho(\Graph) \geq 2J \cdot \frac{m-1}{m} \cdot \frac{n^2j^2}{(J + nj)^2}.
\end{displaymath}
Fixing $m$ and $J$, we consider any sequence of such graphs for which $j \rightarrow 0$ but simultaneously $nj \rightarrow \infty$, then $\frac{n^2j^2}{(J + nj)^2} \rightarrow 1$ and $\diameter(\Graph) \rightarrow 2J$. We can thus, given any $\varepsilon>0$, find such a firework graph $\Graph_{m,\varepsilon}$ for which $2J < \diameter (\Graph_{m,\varepsilon}) < 2J + \varepsilon$ and
\begin{displaymath}
    \diameter (\Graph_{m,\varepsilon}) \geq \rho (\Graph_{m,\varepsilon}) \geq \frac{m-1}{m}\cdot \diameter (\Graph_{m,\varepsilon}) - \varepsilon.
\end{displaymath}
From this we can extract a sequence of graphs satisfying the conclusion of Theorem~\ref{thm:diameter}.
\end{example}

For our third example we consider the flower graphs which will play a role in the lower bound in Theorem~\ref{thm:main}. We will calculate mean distance only in the case of equilateral flowers, but Lemma~\ref{lem:flower} below gives a comparison with the non-equilateral case.

\begin{example}[Flower graph]
\label{ex:flower}
Given $m \geq 1$ and $L > 0$, take $F_m$ to be the flower graph consisting of $m$ edges of length $\ell = \frac{L}{m}$ each, all glued at both ends at a common central vertex; that is, $F_m$ has exactly one vertex, of degree $2m$.

We start by noting the following formula for the integral of the distance function $\dist(x,y)$ over a circle $\mathcal{C}$ (of length $\frac{L}{m}$), i.e., when both $x$ and $y$ are taken from the same petal, which can be obtained by a direct calculation: parametrising the loop by $[-\frac{L}{2m},\frac{L}{2m}]$, with $x \simeq 0$ and the points $\frac{L}{2m}$ and $-\frac{L}{2m}$, identified in $\mathcal{C}$, representing the antipodal point to $x$, we have
\begin{displaymath}
    \int_{\mathcal{C}} \dist(x,y) \,\dy = \int_0^{\frac{L}{2m}} y\,\dy + \int_{-\frac{L}{2m}}^0 (-y)\,\dy = \frac{L}{4m^2}.
\end{displaymath}
%\begin{displaymath}
%    \int_{\mathcal{C}} \dist(x,y) \dy
%    =\int_0 ^{\frac{L}{2m}} |y-x| \, dy + \int_{-\frac{L}{2m}+x}^0 -y+x \, dy+ \int_{-\frac{L}{2m}}^{-\frac{L}{2m}+x} \frac{L}{2}+y-x+\frac{L}{2m} \, dy =  \frac{L^2}{4m^2}
%\end{displaymath}
This is, in particular, independent of $x$ and $y$; and thus, for all $x \in {\mathcal{C}}$
\begin{equation}
\label{eq:rho-circle}
    \rho ({\mathcal{C}}) = \rho_{\mathcal{C}} (x) = \frac{L}{4m^2} = \frac{|{\mathcal{C}}|}{4m}.
\end{equation}
%$\rho_{Circle}(x)=\frac{1}{l}(\frac{l^2}{4m^2})=\frac{l}{4m^2}$
We now return to the flower. For convenience, for the rest of the example we will parametrise the edges as follows: we consider each loop (or \emph{petal}) of the flower at having a dummy vertex at the far point from the central vertex, and identify the loop with $[-\frac{L}{2m},\frac{L}{2m}]$, where $0$ corresponds to the central vertex, and the points $-\frac{L}{2m}$ and $\frac{L}{2m}$ are identified at the dummy vertex.

We can now calculate $\rho_{F_m}(x)$ for a given $x \in F_m$. With the above parametrisation of the edges,
\begin{displaymath}
\begin{aligned}
    \rho_{F_m}(x) &=\frac{1}{L}\left[\int_{\mathcal{C}} \dist(x,y) \dy + 2(m-1)\int_0 ^{\frac{L}{2m}} \big(x+y\big) \,\dy\right]\\
    &=\frac{1}{L}\left(\frac{L^2}{4m^2}+2(m-1)\frac{L(L+4mx)}{8m^2}\right).
\end{aligned}
\end{displaymath}
%$$\rho_F(x)=\frac{1}{l}(\int_C \dist(x,y) \, dy + 2(m-1)\int_0 ^{\frac{l}{2m}} x+y \, dy=\frac{1}{l}(\frac{l^2}{4m^2}+2(m-1)\frac{l(l+4mx)}{8m^2}))$$
Integrating over $x \in [-\frac{L}{2m},\frac{L}{2m}]$ and using the symmetry of the graph finally yields
\begin{displaymath}
    \rho(F_m) = \frac{L}{4m}\left(\frac{2m-1}{m}\right).
\end{displaymath}

Let us consider how $\rho$ behaves as a function of $m \geq 1$. If the total length $L$ is fixed, then, since $x \mapsto \frac{L}{4x}(\frac{2x-1}{x})$ attains a global maximum for $x>0$ between $0$ and $1$, then decreasing to $0$ as $x \to \infty$, it follows that $\rho(F_m)$ attains its maximum at $m=1$, is strictly monotonically decreasing in $m \geq 1$, and $\rho(F_m) \to 0$ as $m \to \infty$.

If instead we fix the length of each edge $\ell = \frac{L}{m}$, we have
\begin{displaymath}
    \rho(F_m)=\frac{\ell(2m-1)}{4m} \longrightarrow \frac{\ell}{2}
\end{displaymath}
as $m \to \infty$.
\end{example}

We finish this section with an optimisation result for flower graphs, which will be needed in the proof of the lower bound in \eqref{eq:bounds-1} in Section~\ref{sec:surgery}. While elementary, its proof is necessarily somewhat technical.

\begin{lemma}
\label{lem:flower}
For any flower graph $F$ of total length $L$, on $m$ edges, we have
\begin{displaymath}
	\rho (F) \geq \left(\frac{2m-1}{4m^2}\right)L,
\end{displaymath}
with equality if and only if $F$ is an equilateral flower on $m$ edges of length $L/m$ each.
\end{lemma}

\begin{proof}
Fix a flower graph $F$ with edge lengths $\ell_1,\ldots,\ell_m$, so that $\sum_{i=1}^m \ell_i = L$, and denote by $0$ the central vertex. We start by recalling from \eqref{eq:rho-circle} that, if $\mathcal{C}$ is a loop of length $\ell > 0$, then $\rho_{\mathcal{C}} (x) = \frac{\ell}{4}$ for all $x \in \mathcal{C}$, and so
\begin{displaymath}
    \rho (\mathcal{C})=\frac{1}{\ell}\int_{\mathcal{C}} \frac{\ell}{4} \dx =\frac{\ell}{4}
\end{displaymath}
as well. Now a similar calculation to the one for the equilateral flower graph yields that, for any given $x \in F$, supposing that $x \in e_i$ and $|e_i| = \ell_i$,
\begin{displaymath}
%\begin{aligned}
    \rho_{F}(x)=\frac{1}{L}\left(\frac{\ell_{i}^2}{4}+\sum_{j \neq i}\ell_j\left(\dist(x,0)+\frac{\ell_j}{4}\right) \right),
%\end{aligned}
\end{displaymath}
where $\dist(x,0)$ represents the distance between $x$ and the central vertex $0$.
Integrating over the corresponding edge $e_i$ gives
\begin{displaymath}
    \int_{e_i} \rho_F(x) \dx = \frac{\ell_i}{L}\left(\frac{\ell_{i}^2}{4}+\sum_{j \neq i}\ell_j\left(\frac{\ell_i}{4}+\frac{\ell_j}{4}\right) \right),
\end{displaymath}
whence
\begin{equation}
\label{eq:flower-general-rho}
    \rho(F) = \frac{1}{4L^2}\sum_{i=1}^m \left(\ell_{i}^3+\sum_{j \neq i}\ell_i \ell_j(\ell_i + \ell_j)\right).
\end{equation}
%If we suppose the flower graph has only two petals $(l_1,l_2)$ we get, recalling that in this case $l_1+l_2=L$:
%$$\rho (F')=\frac{1}{4L^2}(l_1^3+l_2^3+2l_1l_2L)$$
%This expression has a minimum (since the graph is a compact set and the function is continuous), which is seen to be attained when $l_1=l_2=\frac{L}{2}$.
%Now if we iterate this in any two petals then we get that the minimum of the function $\rho (F_L)$, that exists by the same argument as before, is attained when $l_i=\frac{L}{n}, \forall i \in \{1,2,...,n\}$
We will now examine what happens if we perturb the lengths of two petals of $F$, without loss of generality $e_1$ and $e_2$, in such a way that the overall length remains constant: we will, in effect, consider $\frac{\partial \rho}{\partial \ell_1}$ under the assumption that $\ell_3,\ldots,\ell_m$ are constant, and $\ell_2 = L - \ell_1 - \sum_{i=3}^m \ell_i$. To this end we first multiply out the constant term $4L^2$ and isolate the parts of \eqref{eq:flower-general-rho} that depend on $\ell_1$ and $\ell_2$:
\begin{displaymath}
\begin{aligned}
   \ 4L^2\rho(F) &=\sum_{i=1}^m \Bigg[\ell_{i}^3+\sum_{\substack{ j \neq i}}\ell_i \ell_j(\ell_i + \ell_j)\Bigg] \\
   &=\ell_1^3+\sum_{j \neq 1} \ell_1\ell_j(\ell_1+\ell_j)+\ell_2^3+\sum_{j \neq 2} \ell_2\ell_j(\ell_2+\ell_j)+\sum_{i=3}^m \Bigg[\ell_{i}^3+\sum_{\substack{ j \neq i}}\ell_i \ell_j(\ell_i + \ell_j)\Bigg]\\
    &=\ell_1^3+2\sum_{\substack{j \neq 1\\j \ge 2 }} \ell_1\ell_j(\ell_1+\ell_j)+\ell_2^3+2\sum_{\substack{j \neq 2\\j \ge 3}} \ell_2\ell_j(\ell_2+\ell_j)+\underbrace{\sum_{i=3}^m \Bigg[\ell_{i}^3+\sum_{\substack{j \neq i\\j \ge 3 }}\ell_i \ell_j(\ell_i + \ell_j)\Bigg]}_{=:d}, %\\
\end{aligned}
\end{displaymath}
where since the latter sum does not depend on $\ell_1$ or $\ell_2$, we may treat it as a constant $d$ and subtract it from the total; we will thus work with the quantity $D:=4L^2\rho(F)-d$, which we will attempt to minimise in function of $\ell_1$. To simplify the calculation, we will set $B:=\sum_{i=3}^m \ell_i$ and $C:=\ell_1+\ell_2$, which we recall we are treating as constant; we also recall that $L=\sum_{i=1}^m \ell_i$, and in particular $L = B + C$. We may then calculate
\begin{displaymath}
\begin{aligned}
  D &=\ell_1^3+\ell_2^3+2\ell_1\ell_2(\ell_1+\ell_2)+2\sum_{i=3}^m \left[ \ell_2\ell_i(\ell_2+\ell_i)+\ell_1\ell_i(\ell_1+\ell_i)\right] \\
   &=\ell_1^3+\ell_2^3+2\ell_1\ell_2C+2C\sum_{i=3}^m \ell_i^2+2(\ell_1^2+\ell_2^2)\sum_{i=3}^m \ell_i \\
   &=\ell_1^3+\ell_2^3+2\ell_1\ell_2C+2C\sum_{i=3}^m \ell_i^2+2(\ell_1^2+\ell_2^2)(L-C).
\end{aligned}
\end{displaymath}
Writing $\ell_2 = L - B - \ell_1$, we may simplify this to
\begin{displaymath}
  D =\ell_1^3+(C-\ell_1)^3+2\ell_1(C-\ell_1)C+2(L-B)\sum_{i=3}^m l_i^2+2(\ell_1^2+(C-\ell_1)^2)B;
\end{displaymath}
since $L-B$ does not depend on $\ell_1$, under our assumptions we may rewrite $D$ as
\begin{displaymath}
    D = (4B+C)\ell_1^2 - (4BC+C^2)\ell_1 + \text{constant indep.\ of }\ell_1,
\end{displaymath}
which is quadratic in $\ell_1$. In particular, $\frac{\partial D}{\partial \ell_1}$ (itself equal to $4L^2\frac{\partial \rho}{\partial \ell_1}$) is zero if and only if
%under our assumptions $\frac{\partial D}{\partial \ell_1}$ (itself equal to $4L^2\frac{\partial \rho}{\partial \ell_1}$) is equal to $\frac{\partial h}{\partial \ell_1}$, where we define $h$ by discarding the sum from $D$ which does not depend on $\ell_1$:
%\begin{displaymath}
%\begin{aligned}
%    h(\ell_1) &:= \ell_1^3+(C-\ell_1)^3+2\ell_1(C-\ell_1)C+2(\ell_1^2+(C-\ell_1)^2)B\\
%    &= C^3-C^2\ell_1+C\ell_1^2+4B\ell_1^2-2BC\ell_1+2BC^2,
%\end{aligned}
%\end{displaymath}
%which is quadratic in $\ell_1$. A simple calculation now shows that $\frac{\partial h}{\partial \ell_1} = 0$ if and only if
\begin{displaymath}
    \ell_1 = \frac{C}{2} = \frac{L-B}{2};
\end{displaymath}
moreover, this represents the unique global minimum of $D$ (and thus of $\rho$) for $\ell_1 \in [0,C]$.
%Now we need to discover the solution to $ (\frac{\partial \rho}{\partial l_2} =)\frac{\partial \rho}{\partial l_1}=0$.
%\begin{align*} 
%    \frac{\partial \rho}{\partial l_1}=-c^2+2cl_1+8Bl_1-4cB=0 \\  
%    (2c+8B)l_1=c^2+4cB \\
%    l_1=\frac{c(c+4B)}{2(c+4B)}=\frac{c}{2}=\frac{L-B}{2} \\
%\end{align*}  

Summarising, among all flowers on $m$ petals of given total length $L$ and fixed edge lengths $\ell_3,\ldots,\ell_m$, $\rho$ is minimised exactly when $\ell_1 = \ell_2$; moreover, holding all other lengths fixed, any flower graph for which $\ell_1 \neq \ell_2$ will have strictly greater $\rho$ than the one for which there is equality.

Hence, starting from any given non-equilateral flower $F$ on $m$ edges and iteratively applying this argument to pairs of its petals, we obtain a sequence of flower graphs with strictly decreasing $\rho$, whose edge lengths will converge in the limit to those of the equilateral flower. Since for fixed $L>0$ and $m\geq 1$, $\rho$ is obviously continuous with respect to the $m$ edge lengths (being a polynomial, see \ref{eq:flower-general-rho}), which are drawn from a compact subset of $[0,L]^m$, it is immediate that $\rho$ must attain its unique global minimum among all flowers of length $L>0$ on $m$ edges, at the equilateral flower.
\end{proof}

\section{Surgery principles}
\label{sec:surgery}

In this section we list and prove two basic \emph{surgery principles} for $\rho$, which we then use to prove our main bounds \eqref{eq:bounds-1} (including the characterisation of equality). These principles, which give sufficient conditions under which making a ``local'' change to a graph $\Graph$ to create a new graph $\widetilde\Graph$ (e.g. cutting through a vertex, lengthening an edge, attaching or removing a pendant graph) will raise or lower $\rho$; they are generally inspired by, and similar to, such surgery principles for eigenvalues of the Laplacian with standard vertex conditions (see \cite{BerKenKur19}). There will certainly be many more such principles beyond these two, but these are the two needed to prove \eqref{eq:bounds-1}.

The first, cutting through vertices, has a long history in the context of Laplacian eigenvalues (going back at least to \cite[Theorem~1]{KurMalNab13} but implicit in the proof of \cite[Th\'eor\`eme~3.1]{Nic87}). However, since we will not be working with a variational characterisation of $\rho$, the proof is completely different, and relies on comparing the respective mean distance functions $\rho_\Graph$ and $\rho_{\widetilde\Graph}$ in a suitable way (but cf.\ Remark~\ref{rem:no-variation}). The second, the key to the upper bound in \eqref{eq:bounds-1}, is a somewhat generalised version of the principle of ``unfolding pendant edges'', see \cite[Theorem~3.18(4)]{BerKenKur19}, but with rather different hypotheses on the pendant graph being ``unfolded'' (see also Figure~\ref{fig:unfolding-edges}).

\begin{theorem}[Surgery principles for mean distance]
\label{thm:surgery}
Let $\Graph$ be a compact metric graph.
\begin{enumerate}
\item Suppose $\widetilde\Graph$ is formed from $\Graph$ by cutting through a vertex (see \cite[Section~3.1]{BerKenKur19}). %Then $\rho (\widetilde\Graph) \geq \rho (\Graph)$, with equality if and only if the cut was trivial, i.e., $\widetilde\Graph = \Graph$.
Then $\rho (\widetilde\Graph) > \rho (\Graph)$ unless $\widetilde\Graph = \Graph$ (i.e., the cut is the trivial cut, where the cut vertex is unchanged).
%\item Suppose $\widetilde\Graph$ is formed from $\Graph$ by (strictly) lengthening an edge. Then $\rho(\widetilde\Graph) > \rho (\Graph)$.
\item Suppose $\HGraph$ is a pendant subgraph of $\Graph$, that is, $\HGraph \subset \Graph$ and there exists $\mv\in \mV(\Graph)$ such that $\partial\HGraph := \HGraph \cap \overline{\Graph \setminus \HGraph} = \{\mv\}$. Suppose $\widetilde\Graph$ is formed from $\Graph$ by replacing $\HGraph$ with another graph $\widetilde \HGraph$ in such a way that $\widetilde\Graph \setminus \widetilde \HGraph = \Graph \setminus \HGraph$. If the graphs $\HGraph$ and $\widetilde \HGraph$ satisfy, cumulatively,
\begin{enumerate}
\item $|\HGraph| = |\widetilde \HGraph|$,
\item $\rho_{\widetilde \HGraph}(\mv) \geq \rho_\HGraph (\mv)$, and
\item $\rho(\widetilde \HGraph) \geq \rho(\HGraph)$,
\end{enumerate}
then $\rho (\widetilde\Graph) \geq \rho (\Graph)$. The inequality is strict if $\rho_{\widetilde \HGraph} (\mv) > \rho_\HGraph (\mv)$, or if $\rho (\widetilde \HGraph) > \rho (\HGraph)$.
\end{enumerate}
\end{theorem}

\begin{example}
\label{ex:unfolding}
Suppose $\HGraph = \widetilde \HGraph = e_1 \cup e_2$, but the point of attachment to $\Graph \setminus \HGraph = \widetilde\Graph \setminus \widetilde \HGraph$ is changed, as depicted in Figure~\ref{fig:unfolding-edges}:
\begin{figure}[ht]
\centering
\begin{tikzpicture}[scale=0.9]
\draw[thick] (-1.5,0) -- (1.5,0);
\draw[thick] (-0.75,1.3) -- (0.75,-1.3);
\draw[thick] (-0.75,-1.3) -- (0.75,1.3);
\draw[thick,blue] (-1.5,0) -- (0,0);
\draw[thick,blue] (-0.75,1.3) -- (0,0);
\draw[fill,blue] (-1.5,0) circle (1.5pt);
\draw[fill,blue] (0,0) circle (1.5pt);
\draw[fill] (1.5,0) circle (1.5pt);
\draw[fill,blue] (-0.75,1.3) circle (1.5pt);
\draw[fill] (-0.75,-1.3) circle (1.5pt);
\draw[fill] (0.75,1.3) circle (1.5pt);
\draw[fill] (0.75,-1.3) circle (1.5pt);
\node at (-1.5,0) [anchor=north east] {{$e_1$}};
\node at (-0.75,1.3) [anchor=north east] {{$e_2$}};
\node at (-1.8,0.9) [anchor=east] {{$\HGraph$}};
\draw[-{Stealth[scale=0.5,angle'=60]},line width=2.5pt] (2.5,0) -- (3.5,0);
\draw[thick] (6.25,-1.3) -- (7.75,1.3);
\draw[thick] (7.75,-1.3) -- (7,0) -- (8.5,0);
\draw[thick,blue] (4.4,1.5) -- (7,0);
\draw[fill,blue] (7,0) circle (1.5pt);
\draw[fill,blue] (4.4,1.5) circle (1.5pt);
\draw[fill,blue] (5.7,0.75) circle (1.5pt);
\node at (5.7,0.65) [anchor=north east] {{$\widetilde\HGraph$}};
%\node at (5.05,1.025) [anchor=south] {{$e_1$}};
%\node at (6.35,0.375) [anchor=south] {{$e_2$}};
\node at (5.25,1) [anchor=south] {{$e_1$}};
\node at (6.45,0.35) [anchor=south] {{$e_2$}};
\draw[fill] (6.25,-1.3) circle (1.5pt);
\draw[fill] (7.75,-1.3) circle (1.5pt);
\draw[fill] (7.75,1.3) circle (1.5pt);
\draw[fill] (8.5,0) circle (1.5pt);
\end{tikzpicture}
\caption{``Unfolding pendant edges'': replacing the configuration of $e_1 \cup e_2$ in $\HGraph$ by the one in $\widetilde\HGraph$ increases $\rho$.}
\label{fig:unfolding-edges}
\end{figure}
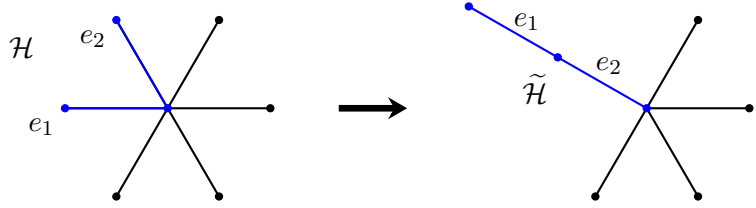
Then both subgraphs may be identified with some interval $[0,\ell]$, $\ell>0$, with $\rho_\HGraph(x)=\rho_{\widetilde{\HGraph}}=\frac{x^2}{\ell}-x+\frac{\ell}{2}$ (see \eqref{eq:rho-interval}). In particular, since $\frac{\ell}{2} > \frac{x^2}{\ell}-x+\frac{\ell}{2}$ for any $x \in (0,\ell)$, and since in $\widetilde{\HGraph}$ we are taking the point of attachment as $\mv\simeq 0$ (whereas in $\HGraph$ the point of attachment will be somewhere in $(0,\ell)$), it follows that $\rho_{\widetilde{\HGraph}}(\mv) = \frac{\ell}{2} > \rho_\HGraph (\mv)$ as long as $|e_1|, |e_2| \neq 0$. We conclude from Theorem~\ref{thm:surgery}(2) that in this case $\rho(\widetilde{\Graph})>\rho(\Graph)$.
\end{example}

{
More difficult is the issue of lengthening an edge, that is, if $\widetilde\Graph$ is formed from $\Graph$ by increasing the length of a given edge $e$ (and thus, equally, the total length of the graph), holding the graph topology and all other edge lengths constant, do we have $\rho (\widetilde\Graph) \geq \rho (\Graph)$? A corresponding statement holds for the standard Laplacian eigenvalues (see \cite[Corollary~3.12(1)]{BerKenKur19}). In the related case of \emph{attaching a pendant graph} at a vertex (see \cite[Definition~3.9]{BerKenKur19}), if one attaches a sufficiently short pendant edge at a vertex $\mv$, then $\rho$ will actually decrease as long as $\rho_\Graph (\mv) < \rho (\Graph)$. This also provides an example whereby lengthening a (very short) edge will lower $\rho$ (our thanks go to Noah Kravitz for pointing out this principle to us).

\begin{proposition}
\label{prop:pendant-derivative}
Suppose the family of graphs $\Graph_\ell$ is formed from $\Graph$ by attaching a pendant edge of length $\ell$ to $\Graph$ at a vertex $\mv$. Then the function $\ell \mapsto \rho (\Graph_\ell)$ is a $C^1$-function of $\ell \geq 0$, and
\begin{displaymath}
	\frac{\d}{\d \ell} \rho (\Graph_\ell)\Big|_{\ell = 0} = \frac{2}{|\Graph|}\big[ \rho_\Graph (\mv) - \rho(\Graph)\big].
\end{displaymath}
\end{proposition}

Note that this result immediately applies, \emph{mutatis mutandis}, if we lengthen any already existing pendant edge of $\Graph$: we obtain a ``Hadamard-type'' formula for the derivative at $\ell = \ell_0$ given by
\begin{displaymath}
	\frac{\d}{\d \ell} \rho (\Graph_\ell)\Big|_{\ell = \ell_0} = \frac{2}{|\Graph|}\big[ \rho_\Graph (\mv_0) - \rho(\Graph)\big],
\end{displaymath}
where $\mv_0$ is the degree-one vertex at the end of the edge. This can be obtained as an immediate consequence of Proposition~\ref{prop:pendant-derivative} by attaching a short edge to $\mv_0$ (turning the latter into a dummy vertex), or else by making a trivial adaptation of the proof of the proposition.
}

\begin{proof}[Proof of Theorem~\ref{thm:surgery}]
    In this proof it will be important to distinguish the distance functions on different graphs; thus we will write $\dist_{\Graph}$ for the distance function on $\Graph$, and so on.
    
    (1) Fix any two $x,y \in \widetilde\Graph$ which are not in the image of the cut vertex in $\Graph$, and recall that by definition
%    $$\dist_{\widetilde{\Graph}} (x,y)=\inf\{|\mathcal{P}|: \mathcal{P} \in C([0,1],\widetilde{\Graph}): \mathcal{P}(0)=x ,\, \mathcal{P}(1)=y\}$$
    $\dist_{\widetilde\Graph}(x,y)$ is the infimum of the lengths of all paths between $x$ and $y$ in $\widetilde\Graph$. Canonically identifying $x$ and $y$ with points in $\Graph$, we see that the set of paths between $x$ and $y$ in $\Graph$ can only be larger than its counterpart in $\widetilde\Graph$; thus the infimum is lower,
    \begin{displaymath}
        \dist_\Graph (x,y) \leq \dist_{\widetilde{\Graph}} (x,y).
    \end{displaymath}
    Since this holds for all $x,y$ outside a finite set, the desired inequality for $\rho$ follows from the definition.
%    $$\rho(\Graph)=\frac{1}{|\Graph|^2} \int_\Graph \int_\Graph \dist_\Graph (x,y) \,\dx\,\dy \leq \frac{1}{|\Graph|^2} \int_\Graph \int_\Graph \dist_{\widetilde{\Graph}}(x,y) \,\dx \,\dy=\rho(\widetilde{\Graph}).$$

    For the strict inequality, suppose $\mv \in \Graph$ is the vertex being cut, and $\mv_1,\mv_2 \in \widetilde\Graph$ are two distinct vertices obtained from the cut. Take $\varepsilon>0$ to be less than a third of the minimum edge length in $\Graph$ (and thus in $\widetilde\Graph$ as well). There will necessarily exist edges $e_1 \sim \mv_1$ and $e_2 \sim \mv_2$ in $\widetilde\Graph$, which are adjacent at $\mv$ in $\Graph$, but not at $\mv_1 \neq \mv_2$ in $\widetilde\Graph$.     Then, in $\Graph$, for all $x \in e_1 \cap B_\varepsilon (\mv)$ and $y \in e_2 \cap B_\varepsilon(\mv)$, $\dist_\Graph (x,y) \leq 2\varepsilon$.
    
    Now consider the shortest path from $x$ to $y$ in $\widetilde\Graph$. There are two possibilities: either it runs through an entire edge of $\widetilde\Graph$, in which case it must have length at least $3\varepsilon$, or it is entirely contained in the union of $e_1$ and $e_2$. Since $e_1$ and $e_2$ are not adjacent at $\mv_1 \neq \mv_2$, in this case they must be adjacent at their respective other endpoints, say at the common vertex $w \in \widetilde\Graph$. Since $|e_1|,|e_2| \geq 3\varepsilon$ and $\dist_{\widetilde\Graph}(x,\mv_1) < \varepsilon$, $\dist_{\widetilde\Graph}(y,\mv_2) < \varepsilon$, it follows that $\dist_{\widetilde{\Graph}}(x,w), \dist_{\widetilde{\Graph}}(y,w) \geq 2\varepsilon$, and thus $\dist_{\widetilde{\Graph}} (x,y) \geq 4\varepsilon$ in this case.

    At any rate, we obtain that $\dist_{\Graph} (x,y) < 2\varepsilon$, while $\dist_{\widetilde{\Graph}} \geq 3\varepsilon$, for all such $x,y$. Integrating over all $x \in e_1 \cap B_\varepsilon (\mv)$ and $y \in e_2 \cap B_\varepsilon (\mv)$ (which each form a set of measure $\varepsilon$) leads to the strict inequality
    \begin{displaymath}
        \int_\Graph \int_\Graph \dist_\Graph (x,y) \,\dx\,\dy < \int_\Graph \int_\Graph \dist_{\widetilde{\Graph}}(x,y) \,\dx\,\dy,
    \end{displaymath}
    and thus $\rho (\Graph) < \rho (\widetilde\Graph)$.

    (2) Note that $|\Graph|=|\widetilde{\Graph}|$, so it suffices to prove that
    \begin{displaymath}
        \int_{\widetilde{\Graph}}\int_{\widetilde{\Graph}} \dist_{\widetilde{\Graph}}(x,y) \,\dx\,\dy \geq \int_\Graph\int_\Graph \dist_{\Graph}(x,y)\, \dx\,\dy.
    \end{displaymath}
    We divide the former integral into three parts, each of which will be analysed separately:
    \begin{displaymath}
        \int_{\widetilde{\Graph}}\int_{\widetilde{\Graph}} \dist_{\widetilde{\Graph}}(x,y)\, \dx \, \dy = \Big( \underbrace{\int_{\widetilde{\Graph} \setminus \widetilde{\HGraph}}\int_{\widetilde{\Graph} \setminus \widetilde{\HGraph}}}_{({\rm i})} + 2\underbrace{\int_{\widetilde{\Graph} \setminus \widetilde{\HGraph}}\int_{\widetilde{\HGraph}}}_{({\rm ii})}+\underbrace{\int_{\widetilde{\HGraph}}\int_{\widetilde{\HGraph}}}_{({\rm iii})}\Big) \dist_{\widetilde{\Graph}}(x,y)\, \dx\, \dy.
    \end{displaymath}
    (i) For the first integral, since for all $x,y \in \widetilde{\Graph} \setminus \widetilde{\HGraph} \simeq \Graph \setminus \HGraph$, the shortest path connecting $x,y$ does not run through the pendant $\widetilde{\HGraph}$ (or $\HGraph$, respectively), with the possible exception of $\mv$,
    \begin{displaymath}
        \dist_{\widetilde{\Graph}}(x,y) = \dist_{\widetilde{\Graph} \setminus \widetilde{\HGraph}}(x,y) = \dist_{\Graph \setminus \HGraph}(x,y)=\dist_\Graph (x,y)
    \end{displaymath}
    for all $x,y \in \widetilde{\Graph} \setminus \widetilde{\HGraph} \simeq \Graph \setminus \HGraph$,
    whence
    \begin{displaymath}
        \int_{\widetilde{\Graph} \setminus \widetilde{\HGraph}}\int_{\widetilde{\Graph} \setminus \widetilde{\HGraph}} \dist_{\widetilde{\Graph}}(x,y)\, \dx \,\dy = \int_{\Graph \setminus \HGraph}\int_{\Graph \setminus \HGraph} \dist_{\Graph}(x,y)\, \dx\,\dy.
    \end{displaymath}
    (ii) For the second term, we are considering points of the form $x \in \widetilde{\Graph} \setminus \widetilde{\HGraph}$, $y \in \widetilde{\HGraph}$; since $\widetilde{\HGraph}$ is a pendant attached at $\mv$, in this case
    \begin{displaymath}
        \dist_{\widetilde{\Graph}}(x,y)=\dist_{\widetilde{\Graph} \setminus \widetilde{\HGraph}}(x,\mv)+\dist_{\widetilde{\HGraph}}(\mv,y)
    \end{displaymath}
    (and similarly for $\HGraph$). Noting that the distance functions in $\Graph \setminus \HGraph$ and $\widetilde\Graph \setminus \widetilde\HGraph$ coincide, we deduce that
    \begin{displaymath}
        \int_{\widetilde{\Graph} \setminus \widetilde{\HGraph}}\int_{\widetilde{\HGraph}} \dist_{\widetilde{\Graph}}(x,y) \,\dy\,\dx = \underbrace{\int_{\widetilde{\Graph} \setminus \widetilde{\HGraph}}\dist_{\widetilde{\Graph} \setminus \widetilde{\HGraph}}(x,\mv)\cdot|\widetilde{\HGraph}|\,\dx}_{=:I}+\underbrace{\int_{\widetilde{\HGraph}} \dist_{\widetilde{\Graph}}(\mv,y)\cdot|\widetilde{\Graph} \setminus \widetilde{\HGraph}|\,\dy}_{=: J}.
    \end{displaymath}
    Using the assumption (a), we may rewrite $I$ as
    \begin{displaymath}
        I = \int_{\Graph \setminus \HGraph}\dist_{\Graph \setminus \HGraph}(x,\mv)\cdot|\HGraph|\,\dx=\int_{\Graph \setminus \HGraph} \int_\HGraph \dist_{\Graph \setminus \HGraph}(x,\mv) \,\dy\,\dx
    \end{displaymath}
    while for $J$ we have, by (a) and (b),
    \begin{displaymath}
    \begin{aligned}    
        J = |\widetilde{\Graph} \setminus \widetilde{\HGraph}|\cdot \int_{\widetilde{\HGraph}}\dist_{\widetilde{\HGraph}}(\mv,y)\,\dy &=|\widetilde{\Graph} \setminus \widetilde{\HGraph}|\cdot |\widetilde{\HGraph}|\rho_{\widetilde{\HGraph}}(\mv) \geq |\Graph \setminus \HGraph|\cdot|\HGraph|\rho_{\HGraph} (\mv) \\
        &=|\Graph \setminus \HGraph|\int_{\HGraph} \dist_{\HGraph} (\mv,y)\,\dy \\
        &= \int_{\Graph \setminus \HGraph} \int_{\HGraph} \dist_{\Graph} (\mv,y) \,\dy\, \dx
    \end{aligned}
    \end{displaymath}
    Putting these two estimates together, we obtain
    \begin{displaymath}
    \begin{aligned}
    \int_{\widetilde{\Graph} \setminus \widetilde{\HGraph}}\int_{\widetilde{\HGraph}} \dist_{\widetilde{\Graph}}(x,y) \,\dy\, \dx &\geq \int_{\Graph \setminus \HGraph} \int_{\HGraph} \dist_{\Graph} (x,\mv)+\dist_{\Graph} (\mv,y) \,\dy \,\dx\\
    &=\int_{\Graph \setminus \HGraph} \int_{\HGraph} \dist_{\Graph} (x,y) \,\dy \,\dx.
    \end{aligned}
    \end{displaymath}
    (iii) Since by (c) we have $\rho(\widetilde{\HGraph})\geq \rho(\HGraph)$, by (a), $|\HGraph|=|\widetilde{\HGraph}|\,(\geq 0)$, and by assumption on $\widetilde{\HGraph}$ and $\HGraph$ as pendants, $\dist_{\widetilde{\Graph}}(x,y) = \dist_{\widetilde{\HGraph}} (x,y)$ for all $x,y \in \widetilde{\HGraph}$ and $\dist_{\Graph} (x,y) = \dist_{\HGraph} (x,y)$ for all $x,y, \in \HGraph$, it is immediate that
    \begin{displaymath}
        \int_{\widetilde{\HGraph}}\int_{\widetilde{\HGraph}}\dist_{\widetilde{\Graph}}(x,y) \,\dx \,\dy \geq \int_{\HGraph}\int_{\HGraph} \dist_{\Graph} (x,y) \,\dx \,\dy.
    \end{displaymath}
    Finally, putting (i), (ii) and (iii) together yields
    \begin{displaymath}
    \begin{aligned}
        \int_{\widetilde{\Graph}}\int_{\widetilde{\Graph} }\dist_{\widetilde{\Graph}}(x,y) \,\dx \,\dy &= \left( \int_{\widetilde{\Graph} \setminus \widetilde{\HGraph}}\int_{\widetilde{\Graph} \setminus \widetilde{\HGraph}} + 2\int_{\widetilde{\Graph} \setminus \widetilde{\HGraph}}\int_{\widetilde{\HGraph}}+\int_{\widetilde{\HGraph}}\int_{\widetilde{\HGraph}}\right) \dist_{\widetilde{\Graph}}(x,y) \,\dx \,\dy \\
        & \geq \left( \int_{\Graph \setminus \HGraph}\int_{\Graph \setminus \HGraph} + 2\int_{\Graph \setminus \HGraph} \int_{\HGraph} + \int_{\HGraph}\int_{\HGraph} \right) \dist_{\Graph}(x,y) \,\dx \,\dy \\
        & = \int_{\Graph} \int_{\Graph} \dist_{\Graph}(x,y) \,\dy \,\dx.
    \end{aligned}
    \end{displaymath}
    For the strict inequality, we note that if $\rho_{\widetilde{\HGraph}}(\mv) > \rho_\HGraph (\mv)$, then the inequality in (ii) is strict, while if $\rho(\widetilde{\HGraph}) > \rho(\HGraph)$, then the inequality in (iii) is strict. In either case, we then have $\rho(\widetilde{\Graph}) > \rho(\Graph)$.
\end{proof}

We can now prove the bounds \eqref{eq:bounds-1}, including the statements about equality.

\begin{proof}[Proof of \eqref{eq:bounds-1}]
For the lower bound, given a graph $\Graph$ which is not already a flower, by Theorem~\ref{thm:surgery}(1) there exists a flower $F$ with the same number of edges, obtained by gluing all vertices of $\Graph$ together, such that $\rho (\Graph) > \rho (F)$. The claim now follows immediately from Lemma~\ref{lem:flower}.

For the upper bound, if $\Graph$ is not already a tree, then by Theorem~\ref{thm:surgery}(1) there exists a tree $T$ having the same total length such that $\rho(T) > \rho(\Graph)$. If $T$ is not already a path graph, we successively apply Example~\ref{ex:unfolding} to pairs of adjacent pendant edges $e_1$, $e_2$, that is, edges $e_1$ and $e_2$ in $T$ sharing a common vertex $\mv$, such that their other vertex is a leaf. After a finite number of steps $T$ is transformed into an interval $I$ of the same total length as $T$ and hence as $\Graph$, with $\rho(I) > \rho(T)$.

To conclude, suppose $\Graph$ is any graph which is not a path graph. If it is not a tree, then we are in the first case, and $\rho(\Graph) < \rho(T) \leq \rho (I)$. If it is a tree, then we are in the second case, and $\rho(\Graph) < \rho(I)$ still. This shows both the inequality, and the sharpness of the inequality at the same time.
%Since if $\Graph$ is not a path graph then it is either not a tree or it is a tree but not a path graph, it follows that in either case $\rho(I) > \rho(\Graph)$.
\end{proof}

{
We finish with the proof that the effect of attaching a pendant edge at a vertex $\mv$ depends on the relation between $\rho_\Graph (\mv)$ and $\rho(\Graph)$.

\begin{proof}[Proof of Proposition~\ref{prop:pendant-derivative}]
We will show that the function
\begin{equation}
\label{eq:ell-function}
	\ell \mapsto \int_{\Graph_\ell}\int_{\Graph_\ell} \dist_{\Graph_\ell} (x,y)\,\dx\,\dy
\end{equation}
is a $C^1$-function of $\ell \geq 0$, with derivative at $\ell = 0$ given by $2|\Graph|\rho_\Graph (\mv)$; the other assertions of the proposition then follow directly upon applying the quotient rule to the function
\begin{displaymath}
	\ell \mapsto \rho (\Graph_\ell) = \frac{1}{|\Graph_\ell|^2}\int_{\Graph_\ell}\int_{\Graph_\ell} \dist_{\Graph_\ell} (x,y)\,\dx\,\dy
\end{displaymath}
(and noting that $\ell \mapsto |\Graph_\ell|$ is a smooth function, with $\frac{\d}{\d \ell} |\Graph_\ell| = 1$ for all $\ell \geq 0$). Now, for fixed $\ell > 0$, denoting by $e = e(\ell)$ the pendant edge of length $\ell$ attached to $\mv$, we have
\begin{displaymath}
%\begin{aligned}
	\int_{\Graph_\ell}\int_{\Graph_\ell} \dist_{\Graph_\ell} (x,y)\,\dx\,\dy = \left(\int_\Graph \int_\Graph
	+ 2\int_\Graph \int_e + \int_e \int_e \right)\dist_{\Graph_\ell} (x,y)\,\dx\,\dy.
%\end{aligned}
\end{displaymath}
Using the fact that $e$ is a pendant (and, as a subgraph, identifiable with an interval), we obtain
\begin{displaymath}
	\int_e \int_e \dist_{\Graph_\ell} (x,y)\,\dx\,\dy = \frac{\ell^3}{3}
\end{displaymath}
(see Example~\ref{ex:star}, and in particular \eqref{eq:rho-interval}), as well as
\begin{displaymath}
	\int_\Graph \int_e \dist_{\Graph_\ell} (x,y)\,\dx\,\dy = \ell \cdot \big[\tfrac{\ell}{2} + \rho_\Graph (\mv)\big],
\end{displaymath}
since the mean distance from an arbitrarily chosen $x \in e$ to $\mv$ will be $\frac{\ell}{2}$, the mean distance from $\mv$ to an arbitrary $y \in \Graph$ is, by definition, $\rho_\Graph (\mv)$, and this mean over $\Graph$ needs to be integrated over the edge $e$ of length $\ell$ to obtain the total integral value. Putting these together, it follows that for any $\ell \geq 0$ the derivative of the function \eqref{eq:ell-function} exists and at zero is given by
\begin{displaymath}
	\lim_{\ell \to 0} \tfrac{1}{\ell} \left( 2\ell \cdot \big[\tfrac{\ell}{2} + \rho_\Graph (\mv)\big] + \tfrac{\ell^3}{3}\right) = 2|\Graph|\rho_\Graph (\mv),
\end{displaymath}
as claimed. The statement of the proposition now follows.
\end{proof}
}

\section{Symmetrisation and doubly connected graphs}
\label{sec:symmetrisation}

In this section we will prove the sharpened upper bound \eqref{eq:bounds-2} for doubly connected graphs, including the characterisation of equality. We recall that a graph $\Graph$ is said to be \emph{doubly (path) connected} if, for any $x,y \in \Graph$, there exist two edgewise disjoint paths $P_1$ and $P_2$ in $\Graph$ connecting $x$ and $y$ (although we allow the paths to intersect at a finite number of vertices); equivalently, no edge of $\Graph$ is a \emph{bridge} whose removal would disconnect $\Graph$.

Unlike for \eqref{eq:bounds-1}, in this case there is a natural proof via symmetrisation using the same basic tool, the \emph{coarea formula}, as Friedlander's symmetrisation-based proof of Nicaise' inequality \cite{Fri05} (see also \cite{BanLev17,BerKenKur17}, and Remark~\ref{rem:why-symmetrisation}). We recall that the coarea formula states that
\begin{equation}
\label{eq:coarea}
    \int_\Omega \varphi |\nabla \psi|\,{\rm d}x = \int_0^\infty \int_{\{x:\psi(x)=t\}}\varphi(x)\,{\rm d}\sigma\, {\rm d}t,
\end{equation}
valid for an integrable nonnegative function $\varphi$ on some domain $\Omega$ (which can be a metric graph) and an absolutely continuous, integrable nonnegative function $\psi$ (for metric graphs, see, e.g., \cite{Maz23} or \cite[Lemma~4.6]{DufKenMug22}).

\begin{lemma}
\label{lem:coarea-distance}
For any compact graph $\Graph$ of total length $L>0$ and any $x \in \Graph$, we have
\begin{displaymath}
	\rho_\Graph (x) = \frac{1}{L}\int_0^{\max_y \dist (x,y)} [\# \{y: \dist (x,y)=t\}]\cdot t\,{\rm d}t
\end{displaymath}
\end{lemma}

\begin{proof}
We apply the coarea formula \eqref{eq:coarea} to the Lipschitz continuous function $\varphi = \psi = \dist (x,\cdot)$. Since its gradient has absolute value $1$ almost everywhere, this yields
\begin{displaymath}
\begin{aligned}
	\rho_{\Graph}(x) &= \frac{1}{L}\int_0^\infty \int_{\{y:\dist (x,y)=t\}} t\,{\rm d}\sigma {\rm d}t\\
	&= \frac{1}{L}\int_0^{\max_y \dist (x,y)} [\# \{y: \dist (x,y)=t\}]\cdot t\,{\rm d}t,
\end{aligned}
\end{displaymath}
as claimed.
\end{proof}

%\begin{example}
%For the loop (or circle) $\mathcal{C}$ of length $L$, by symmetry, we have
%\begin{displaymath}
%    \rho_{\mathcal{C}} (x) = \frac{1}{L}\int_0^{L/2} 2t\,{\rm d}t = \frac{L}{4}
%\end{displaymath}
%for any $x \in \mathcal{C}$. This leads to the known value of $\frac{L}{4}$ for the mean distance of the loop:
%\begin{displaymath}
%    \rho (\mathcal{C}) = \frac{1}{L}\int_{\mathcal{C}} \rho_{\mathcal{C}} (x)\,{\rm d}x
%    = \frac{1}{L} \cdot L \cdot \frac{L}{4} = \frac{L}{4}.
%\end{displaymath}
%The loop is the only graph with the property that $\rho_{\Graph}(x)$ is constant independent of $x$. Curiously, this property will play a central role in the symmetrisation argument.
%\end{example}

We recall (Example~\ref{ex:flower}) that for the loop (or circle) $\mathcal{C}$ of length $L$,
\begin{displaymath}
    \rho (\mathcal{C}) = \rho_{\mathcal{C}}(x) = \frac{L}{4};
\end{displaymath}
{ large classes of examples and some intuition suggest that loops should be the only graphs $\Graph$ with the property that $\rho_{\Graph}(x)$ is constant independent of $x$, although we will not attempt to prove that here. Curiously, precisely this independence} will play a central role in the symmetrisation argument.

\begin{lemma}
\label{lem:doubly-connected-distance}
Let $\Graph$ be a compact, doubly connected graph of length $L>0$. For each $x \in \Graph$,
\begin{displaymath}
	M(x):=\max_{y \in \Graph} \dist (x,y) \leq \frac{L}{2}.
\end{displaymath}
The inequality is strict for at least one $x_0 \in \Graph$ (and thus all $x$ in a small neighbourhood of $x_0$) if $\Graph$ is not a loop.
\end{lemma}

\begin{proof}
Fix $x,y\in\Graph$. Since $\Graph$ is doubly connected, there exist at least two paths $P_1$, $P_2$ in $\Graph$ from $x$ to $y$, which can intersect at at most a finite set of points. It follows that $|P_1|+|P_2| \leq L$ and hence at least one path has length at most $\frac{L}{2}$.

If $M(x)=\frac{L}{2}$ for all $x \in \Graph$, then for any $x \in \Graph$ there exists $y \in \Graph$ such that there are exactly two edgewise disjoint paths $P_1$ and $P_2$ from $x$ to $y$, each of length exactly $\frac{L}{2}$, whose union exhausts $\Graph$. The only possibility is that $\Graph$ is a loop.
\end{proof}

\begin{proof}[Proof of \eqref{eq:bounds-2}]
We will show that, for any $x \in \Graph$, we have
\begin{equation}
\label{eq:rho-pointwise-comparison-circle}
    \rho_\Graph (x) \leq \frac{L}{4},
\end{equation}
that is, the pointwise value of $\rho_\Graph$ is always below the constant value of $\rho_{\mathcal{C}}$ on $\mathcal{C}$. The inequality will follow immediately from \eqref{eq:rho-pointwise-comparison-circle}.

Fix $x \in \Graph$ and write
\begin{displaymath}
    \xi(t):= \xi_x (t) := \# \{y: \dist (x,y)=t\}
\end{displaymath}
for the size of the level set of the corresponding distance function, for any $0 \leq t \leq M := M(x) \leq \frac{L}{2}$. We extend $\xi(t)$ by zero to a function on $[0,\frac{L}{2}]$; then, by the coarea formula \eqref{eq:coarea},
\begin{displaymath}
    \int_0^\frac{L}{2} \xi(t)\,{\rm d}t = \int_0^M \xi(t)\,{\rm d}t = L.
\end{displaymath}
We claim that
\begin{equation}
\label{eq:symmetrisation-target}
    \int_0^{M} \xi (t) \cdot t\,{\rm d}t \leq \int_0^{\frac{L}{2}}2\cdot t\,{\rm d}t = \frac{L^2}{4}.
\end{equation}
To see this, note that, obviously,
\begin{displaymath}
    \int_0^\frac{L}{2} 2\,{\rm d}t = L = \int_0^\frac{L}{2} \xi(t)\,{\rm d}t,
\end{displaymath}
and so
\begin{displaymath}
    0 \leq \int_0^M \xi(t)-2\,{\rm d}t = -\int_M^\frac{L}{2} \xi(t)-2\,{\rm d}t = \int_M^\frac{L}{2} 2 - \xi(t)\,{\rm d}t;
\end{displaymath}
note that the first integral is necessarily nonnegative since the nonnegative functions $\xi$ and $2$ have the same $L^1$-norm on $[0,\frac{L}{2}]$, while $\xi$ is supported only on $[0,M] \subset [0,\frac{L}{2}]$. (The first integral is seen to be zero if and only if $M=\frac{L}{2}$.) It follows directly that
\begin{displaymath}
    \int_0^M (\xi(t)-2)t\,{\rm d}t \leq M\int_0^M \xi(t)-2\,{\rm d}t
    = M\int_M^{\frac{L}{2}} 2-\xi(t)\,{\rm d}t \leq \int_M^{\frac{L}{2}} (2-\xi(t))t\,{\rm d}t.
\end{displaymath}
Rearranging yields \eqref{eq:symmetrisation-target} and thus the pointwise comparison between $\rho_{\Graph}$ and $\rho_{\mathcal C}$. The inequality statement in the theorem now follows immediately from Lemma~\ref{lem:coarea-distance}, since
\begin{displaymath}
    \rho_\Graph (x) = \frac{1}{L} \int_0^M \xi (t)\cdot t\,{\rm d}t \leq \frac{1}{L} \int_0^{\frac{L}{2}}2\cdot t\,{\rm d}t = \frac{L}{4},
\end{displaymath}
which establishes \eqref{eq:rho-pointwise-comparison-circle} for arbitrary $x \in \Graph$.

Finally, if $\Graph$ is not a loop, then there exists some $x_0 \in \Graph$ for which $M(x_0) < \frac{L}{2}$. By a continuity argument, we can find $\delta,\varepsilon > 0$ such that $M(x) \leq \frac{L}{2} - \varepsilon$ %and
%\begin{displaymath}
%    \int_0^{M(x)} \xi_x (t) - 2\,{\rm d}t \geq \varepsilon
%\end{displaymath}
for all $x$ in a $\delta$-neighbourhood of $x_0$. For such $x$, we have
\begin{displaymath}
\begin{aligned}
    \int_0^{M(x)} (\xi_x(t)-2)t\,{\rm d}t &\leq M(x)\int_{M(x)}^{\frac{L}{2}} 2-\underbrace{\xi_x(t)}_{=0}\,{\rm d}t\\
    &= \int_{M(x)}^{\frac{L}{2}} (2-\xi_x(t))t\,{\rm d}t - \int_{M(x)}^{\frac{L}{2}} 2(t-M(x))\,{\rm d}t.
\end{aligned}
\end{displaymath}
Since
\begin{displaymath}
    \int_{M(x)}^{\frac{L}{2}} 2(t-M(x))\,{\rm d}t
    = \int_{0}^{\frac{L}{2}-M(x)} 2t\,{\rm d}t \geq \int_0^\varepsilon 2t\,{\rm d}t = \varepsilon^2,
\end{displaymath}
we can thus sharpen \eqref{eq:symmetrisation-target} to
\begin{displaymath}
        \int_0^{M(x)} \xi_x (t) \cdot t\,{\rm d}t \leq \int_0^\frac{L}{2} 2t\,{\rm d}t - \varepsilon^2 = \frac{L^2}{4}-\varepsilon^2
\end{displaymath}
and thus obtain
\begin{displaymath}
    \rho_\Graph (x) \leq \frac{L}{4} - \frac{\varepsilon^2}{L}
\end{displaymath}
for all $x$ in a $\delta$-neighbourhood of some $x_0 \in \Graph$; in particular, the total measure of all such $x \in \Graph$ is at least $2\delta$. It follows that the inequality $\rho(\Graph) \leq \frac{L}{4}$ must be strict.
\end{proof}

\begin{remark}
\label{rem:why-symmetrisation}
Curiously, it is not clear how this symmetrisation idea could be adapted to simply connected graphs and the interval: on the loop the mean distance function is constant, so we obtain a pointwise comparison between $\rho_\Graph (x)$ and the constant value $\rho_{\mathcal{C}}(x)=\frac{L}{4}$. To use such an argument compare an arbitrary graph with an interval (path graph) $I$, where $\rho_I(x)$ is \emph{not} constant (see \eqref{eq:rho-interval}), we would first need to find a way of associating, for each $x \in \Graph$, some unique $y_x \in I$ such that then $\rho_\Graph (x) \leq \rho_{I}(y_x)$. It is not at all clear how this should work.
\end{remark}

\section{A variational comparison method}
\label{sec:test}

This short section is devoted to the proof of Theorem~\ref{thm:direct-lower-bound}. We first fix any point, without loss of generality a vertex, $\mv\in \Graph$ such that
\begin{displaymath}
    \rho_\Graph (\mv) = \frac{1}{|\Graph|}\int_\Graph \dist (x,v) \dx = \rho (\Graph),
\end{displaymath}
whose existence follows directly from the continuity of the function $\rho_\Graph$ and the definition of $\rho (\Graph)$ as mean value of $\rho_\Graph$.

We define
\begin{equation}
\label{eq:dirichlet-one-point}
    \lambda_1 (\Graph, \mv):= \inf_{\substack{0 \neq f \in H^1(\Graph)\\ f(\mv)=0}} \frac{\int_\Graph |f'(x)|^2\,\dx}{\int_\Graph |f(x)|^2\,\dx}
\end{equation}
to be the first eigenvalue of the Laplacian on $\Graph$ with a Dirichlet condition at $\mv$ and standard conditions elsewhere (we allow that $\Graph \setminus \{\mv\}$ be disconnected; in this case the corresponding eigenfunctions may be supported on a proper subset of $\Graph$). Fix any $f \in H^1(\Graph)$ such that $f(\mv)=0$. For $x \in \Graph$ we denote by $P(\mv,x)$ any shortest path in $\Graph$ from $\mv$ to $x$, then from
\begin{displaymath}
    f(x) - f(\mv) = \int_{P(\mv,x)} f'(y) \,\dy
\end{displaymath}
and the Cauchy--Schwarz inequality it follows that
\begin{displaymath}
    |f(x)| \leq \left(\int_{P(\mv,x)} |f'(y)|^2\dy\right)^{1/2} \dist(\mv,x)^{1/2}.
\end{displaymath}
Replacing $P(\mv,x)$ with $\Graph$, taking squares, integrating over $x \in \Graph$ and infimising over $f$ yields
\begin{displaymath}
    1 \leq \lambda_1 (\Graph,\mv) \int_\Graph \dist(\mv,x)\,\dx
    = \lambda_1 (\Graph,\mv) \rho(\Graph) L.
\end{displaymath}
Theorem~\ref{thm:direct-lower-bound} now follows from the following eigenvalue comparison result.

\begin{lemma}
\label{lem:dirichlet-neumann-eigenvalue-comparison}    
Let $\Graph$ be a compact metric graph, let $\mv\in \Graph$ be arbitrary, and let $\lambda_1 (\Graph,\mv)$ be the first eigenvalue of the Laplacian on $\Graph$ with a Dirichlet condition at $\mv$ and standard conditions elsewhere, as defined by \eqref{eq:dirichlet-one-point}. Then
\begin{displaymath}
    \lambda_1 (\Graph,\mv) \leq \mu_2 (\Graph).
\end{displaymath}
\end{lemma}

\begin{proof}
Denote by $\psi \in H^1(\Graph)$ any eigenfunction associated with $\mu_2(\Graph)$, and denote by $\Graph_+$ and $\Graph_-$ any two (disjoint) nodal domains of $\psi$, i.e., connected subsets of $\{x \in \Graph: \psi (x) \neq 0\}$; then a standard variational argument (cf.\ \cite[Proof of Theorem~3.4]{BerKenKur17}) shows that $\mu_2(\Graph) = \lambda_1 (\Graph_+,\partial\Graph_+) = \lambda_1 (\Graph_-,\partial\Graph_-)$ where $\lambda_1 (\Graph_\pm,\partial\Graph_\pm)$ is the first eigenvalue of $\Graph_\pm$ with Dirichlet conditions at every point in $\partial\Graph_\pm$ (and standard conditions at all interior vertices).

Now necessarily $\mv\in \Graph \setminus \Graph_+$ or $\mv\in \Graph \setminus \Graph_-$, without loss of generality $\Graph \setminus \Graph_+$. Then domain monotonicity for Dirichlet eigenvalues implies that
\begin{displaymath}
    \mu_2 (\Graph) = \lambda_1 (\Graph_+,\partial\Graph_+) \geq \lambda_1 (\Graph,\mv),
\end{displaymath}
as claimed.
\end{proof}

\begin{remark}
\label{rem:no-variation}
We finish with some open-ended comments related to the observation that there is no obvious variational characterisation of $\rho$; it would be very interesting if one could be obtained as this would open up the use of more variational techniques to studying it, and would potentially offer insights into the parallels with the variational quantity $\mu_2$. We note that, as is easy to see, for any given $x \in \Graph$ (with $\Graph$ a compact metric graph), the mean distance from $x$ can be characterised as
\begin{displaymath}
    \rho_\Graph (x) = \frac{1}{|\Graph|}\sup \{ \|f\|_{L^1(\Graph)} : f \in H^1(\Graph),\, f(x)=0,\, |f'|=1 \text{ a.e.}\}.
\end{displaymath}
This principle can be used to rephrase slightly some of the arguments presented above, at least those which involve comparing two graphs $\Graph$ and $\widetilde\Graph$ for which there is some kind of natural pointwise correspondence between them. For example, to prove Theorem~\ref{thm:surgery}(1) one can use that there is a canonical identification $\Graph \to \widetilde\Graph$ such that, under this identification, $C(\Graph) \subset C(\widetilde\Graph)$, with preservation of all function norms, and for almost every $x$, due to the weakened continuity requirement in $\widetilde\Graph$,
\begin{displaymath}
    \{ f \in H^1(\Graph),\, f(x)=0,\, |f'|=1 \text{ a.e.}\} \subset \{ f \in H^1(\widetilde\Graph),\, f(x)=0,\, |f'|=1 \text{ a.e.}\},
\end{displaymath}
whence $\rho_{\Graph} (x) \leq \rho_{\widetilde\Graph} (x)$ for a.e.\ $x$. Integrating over $x$ and using that $|\Graph| = |\widetilde\Graph|$ yields the result. This argument mirrors more closely the corresponding argument for $\mu_2$.
\end{remark}

\bibliographystyle{plain}
%\bibliography{../../../referenzen/literatur.bib} 

\end{document}